\documentclass{article}
\usepackage[T1]{fontenc}
\usepackage[utf8]{inputenc}
\usepackage[english]{babel}

\title{Exponent Preserving Subgroups\\of the Finite Simple Groups}
\author{Andrea Pachera \\
	\multicolumn{1}{p{.5\textwidth}}{\centering\small\texttt{pak.ska@gmail.com}}}

\usepackage{ifpdf}

\usepackage[stretch=10]{microtype}
\usepackage{mathtools}%loads amsmath
\mathtoolsset{centercolon}
\usepackage{amsthm}
\usepackage{amssymb}%loads amsfonts
\usepackage{mathabx}%allows \notdivides
\usepackage{bbm}
\usepackage{bm}
\usepackage{calc}%allows \widthof
\usepackage{enumitem}
	\setlist[itemize]{labelindent=1.5\parindent,leftmargin=*}
	\setlist[enumerate]{label=(\roman*),font=\normalfont}
\usepackage{nicefrac}
\usepackage{braket}%allows \Set{ | }
\usepackage{booktabs}
\usepackage{longtable}
\usepackage{caption}
\usepackage{hyperref}%order! hyperref --> amsthm&(geometry-->todonotes) --> cleveref
\hypersetup{
    colorlinks,
    linkcolor={red!50!black},
    citecolor={blue!50!black},
    urlcolor={blue!80!black}
}

\ifpdf
\usepackage[pdftex]{geometry}
\else
\usepackage{geometry}
\fi
\usepackage[colorinlistoftodos, textwidth=\marginparwidth, color=yellow,textsize=footnotesize]{todonotes}
\usepackage{csquotes}
\usepackage[style=numeric,sortcites,backend=bibtex,bibencoding=ascii]{biblatex}%arXiv uses bibtex
\renewcommand{\sp}[2]{\mathrm{Sp}_{#1}(#2)}
\newcommand{\psp}[2]{\mathrm{PSp}_{#1}(#2)}
\newcommand{\go}[3]{\mathrm{GO}^{#1}_{#2}(#3)}
\newcommand{\so}[3]{\mathrm{SO}^{#1}_{#2}(#3)}

\newcommand{\om}[3]{\Omega^{#1}_{#2}(#3)}
\newcommand{\pom}[3]{\mathrm{P}\Omega^{#1}_{#2}(#3)}
\renewcommand{\k}[3]{\mathrm{K}^{#1}_{#2}(#3)}
\renewcommand{\l}[2]{\mathrm{L}_{#1}(#2)}
\newcommand{\gl}[2]{\mathrm{GL}_{#1}(#2)}

\renewcommand{\sl}[2]{\mathrm{SL}_{#1}(#2)}
\newcommand{\psl}[2]{\mathrm{PSL}_{#1}(#2)}
\newcommand{\m}[1]{\mathrm{M}_{#1}}
\newcommand{\uni}[2]{\mathrm{U}_{#1}(#2)}

\newcommand{\mcl}{\mathrm{McL}}
\newcommand{\hs}{\mathrm{HS}}
\newcommand{\co}[1]{\mathrm{Co}_{#1}}

\newcommand{\ord}[2]{\mathrm{ord}_{#1}(#2)}

\newcommand{\aut}{\mathrm{Aut}\,}
\renewcommand{\epsilon}{\varepsilon}
\renewcommand{\phi}{\varphi}
\newcommand{\f}[1]{\mathbb{F}_{#1}}
\newcommand{\twr}{\wr_{\text{tw}}}

\newcommand{\psylow}{Sylow \(p\text{-subgroup}\) }
\newcommand{\psylows}{Sylow \(p\text{-subgroups}\) }
\newcommand{\dsylow}{Sylow \(2\text{-subgroup}\) }
\newcommand{\dsylows}{Sylow \(2\text{-subgroups}\) }
\newcommand{\psyloww}{Sylow \(p\text{-subgroup}\)}
\newcommand{\dsyloww}{Sylow \(2\text{-subgroup}\)}
\newcommand{\dimn}[1]{#1\text{-dimensional}}
\newcommand{\br}{\par\bigskip}
\DeclarePairedDelimiter{\size}{\lvert}{\rvert}
\DeclarePairedDelimiter{\gen}{\langle}{\rangle}
\theoremstyle{plain}
\newtheorem{theorem}{Theorem}[section]
\newtheorem{proposition}[theorem]{Proposition}
\theoremstyle{definition}
\newtheorem*{theorem-n}{Theorem}

\newtheorem*{notation}{Notation}

\newtheorem*{help}{Acknowledgments}
\allowdisplaybreaks[4]
\usepackage{cleveref}%last

\begin{document}

\maketitle
\begin{abstract}
	Given a group \(G\) denote with \(\exp(G)\) its exponent, which is the least common multiple of the order of its elements. In this paper we solve the problem of finding the finite simple groups having a proper subgroup with the same exponent. For each \(G\) with this property we will give an explicit example of \(H<G\) with \(\exp(G)=\exp(H)\).
\end{abstract}
\section{Introduction}

Given a group \(G\), denote with \(\pi(G)\) the set of prime divisors of \(\size{G}\), with \(\Gamma(G)\) its prime graph, and with \(\exp(G)\) its exponent, i.e. the least common multiple of the orders of its elements.

In the recent years a series of problems have been investigated, related to the existence of a suitable subgroup \(H<G\) preserving some prescribed property of \(G\). For example, Lucchini, Morigi, and Shumyatsky~\cite{al} proved that if \(G\) is finite then it always contains a 2-generated subgroup \(H\) with \(\pi(G)=\pi(H)\), and a 3-generated subgroup \(H\) with \(\Gamma(G)=\Gamma(H)\); Covato~\cite{covato} extended these results to profinite groups; Burness and Covato~\cite{burness} showed which finite simple groups \(G\) contain a proper subgroup \(H\) with \(\Gamma(G)=\Gamma(H)\).

The aim of this work is to find the finite simple groups which have a proper subgroup with the same exponent, and we prove the following result:
\begin{theorem-n}\label{maintheorem}
	The finite simple groups which contain a proper subgroup with the same exponent are the following:
	\begin{enumerate}
		\item the alternating groups \(A_n\) with \(n\ge5\), except when \(n=10\), \(n=p^r\) with \(p\) odd prime, or \(n=p_f+1\) where \(p_f\) is a Fermat prime;
		\item the symplectic groups \(\psp{4}{q}\), except when \(q=3^k\) or \(q=2\);
		\item the symplectic groups \(\psp{2m}{q}\) with \(m\) and \(q\) even, except when \(m=q=2\);
		\item the orthogonal groups \(\pom{}{2m+1}{q}\) with \(m\ge4\) even, except when \(p^a=2m-1\) for some \(a\), where \(q=p^k\);
		\item the orthogonal groups \(\pom{+}{2m}{q}\) with \(m\ge4\) even;
		\item the Mathieu groups \(\m{12},\,\m{24}\);
		\item the Higman-Sims group \(HS\).
	\end{enumerate}
\end{theorem-n}
The proof is based on~\cite{doctabella}, where Table 10.7 can be used to obtain \autoref{table} below, which contains a list of all the possible pairs \((G,M)\) where \(M\) is a maximal subgroup of the finite simple group \(G\) with \(\pi(M)=\pi(G)\). For each of these pairs \((G,M)\) we check whether \(\exp(M)=\exp(G)\), organizing our discussion in the following way:
\begin{enumerate}
	\item in \autoref{casidifficili} we solve the problem for the four infinite families of classical groups of Lie type (a)-(d). Notice that this requires the study of the \psylows of the groups involved, which are discussed in \autoref{capitolopsylow}, in order to compare the exponents of \(G\) and \(M\);
	\item the alternating group (e) is studied in \autoref{alternante};
	\item in the other cases, the possibilities for \(M\) are explicitly listed, so it's easy to make a direct computation: they are studied in \autoref{casifacili}.
\end{enumerate}
\begin{table}[h]
	\centering
	\begin{tabular}{cl@{\hskip 5ex}c@{\hskip 5ex}c}
		\toprule
		&\(G\)	&	\(M\)	&	Remarks	\\
		\midrule
		(a)&\(\psp{2m}{q}\)	&	\(M\unrhd\om{-}{2m}{q}\)	&	\(m,\,q\) even	\\
		(b)&\(\pom{}{2m+1}{q}\)	&	\(M\unrhd\om{-}{2m}{q}\)	&	\(m\) even, \(q\) odd	\\
		(c)&\(\pom{+}{2m}{q}\)	&	\(M\unrhd\om{}{2m-1}{q}\)	&	\(m\) even	\\
		(d)&\(\psp{4}{q}\)	&	\(M\unrhd\psp{2}{q^2}\)	&	\\
		(e)&\(A_c\)	&	\(A_k\unlhd M\le S_k\times S_{c-k}\)	&	\\
		&\(A_6\)	&	\(\l{2}{5}\)	&	\\
		&\(\l{6}{2}\)	&	\(P_1,\,P_5\)	&	\\
		&\(\uni{3}{3}\)	&	\(\l{2}{7}\)	&	\\
		&\(\uni{3}{5}\)	&	\(A_7\)	&	\\
		&\(\uni{4}{2}\)	&	\(2^4\rtimes A_5,\,S_6\)	&	\\
		&\(\uni{4}{3}\)	&	\(\l{3}{4},\,A_7\)	&	\\
		&\(\uni{5}{2}\)	&	\(\l{2}{11}\)	&	\\
		&\(\uni{6}{2}\)	&	\(\m{22}\)	&	\\
		&\(\psp{4}{7}\)	&	\(A_7\)	&	\\
		&\(\sp{6}{2}\)	&	\(S_8\)	&	\\
		&\(\pom{+}{8}{2}\)	&	\(P_1,\,P_3,\,P_4,\,A_9\)	&	\\
		&\(G_2(3)\)	&	\(\l{2}{13}\)	&	\\
		&\(\prescript{2}{}{F_4(2)'}\)	&	\(\l{2}{25}\)	&	\\
		&\(\m{11}\)	&	\(\l{2}{11}\)	&	\\
		&\(\m{12}\)	&	\(\m{11},\,\l{2}{11}\)	&	\\
		&\(\m{24}\)	&	\(\m{23}\)	&	\\
		&\(\hs\)	&	\(\m{22}\)	&	\\
		&\(\mcl\)	&	\(\m{22}\)	&	\\
		&\(\co{2}\)	&	\(\m{23}\)	&	\\
		&\(\co{3}\)	&	\(\m{23}\)	&	\\
		\bottomrule
	\end{tabular}
	\caption{The pairs \((G,M)\) with \(\pi(G)=\pi(M)\).}\label{table}
\end{table}
\begin{notation}
	The notation is fairly standard, with the classical groups denoted in the following way:
	\begin{gather*}
	\l{n}{q}=\psl{n}{q}\le\sl{n}{q}\le\gl{n}{q}\\
	\psp{2n}{q}\le\sp{2n}{q}\\
	\pom{\eta}{n}{q}\le\om{\eta}{n}{q}\le\so{\eta}{n}{q}\le\go{\eta}{n}{q}
	\end{gather*}
	where \(\eta=\pm1\) if \(n\) is even, and it's omitted if \(n\) is odd. It may be omitted also when referring to an unspecified orthogonal group.
	
	\(\k{}{n}{q}\) denotes the kernel of the spinor norm in \(\go{}{n}{q}\), so that \(\k{}{n}{q}\cap\so{}{n}{q}=\om{}{n}{q}\).
	
	\(P_i\) denotes a parabolic subgroup stabilizing a \(\dimn{i}\) subspace.
	
	\(A\twr B\) denotes a twisted wreath product as described in~\cite{wong}.
	
	\(\ord{p}{q}\) denotes the multiplicative order, i.e. it's the minimum \(e\) such that \(q^e\equiv1\bmod p\).
	
	\(a^b\parallel c\) means that \(a^b\mid c\) but \(a^{b+1}\notdivides c\).
\end{notation}
\section{The \texorpdfstring{\psylows}{Sylow p-subgroups}of the classical groups}\label{capitolopsylow}

The four infinite families to study involve the groups \(\psp{2m}{q}\), \(\pom{}{2m+1}{q}\), \(\pom{+}{2m}{q}\), and \(\om{-}{2m}{q}\).

In order to evaluate their exponents, we study their \psylows to get the corresponding power of \(p\) in the factorization of the exponent. In particular, the following hold:
\begin{proposition}
	Take \(G=\sp{2c}{q}\) or \(\go{}{d}{q}\), \(p\) odd with \(p\notdivides q\), \(e=\ord{p}{q}\), and \(p^r\parallel q^e-1\). 	A \psylow of \(G\) is isomorphic to a \psylow of \(\gl{n}{q}\) (if \(e\) is even) or \(\sp{2n}{q}\) (if \(e\) is odd) for some \(n\). In particular:
	\begin{align}
	\label{recapgl}		\exp_p(\gl{n}{q})&=p^{r+v} & ep^v&\le n<ep^{v+1},\\
	\label{recapsp1}	\exp_p(\sp{2n}{q})&=p^{r+v} & 2ep^v&\le 2n<2ep^{v+1}.
	\end{align}
	Since \(p\) is odd, then \(\exp_p(\psp{2c}{q})=\exp_p(\sp{2c}{q})\) and \(\exp_p(\pom{}{d}{q})=\exp_p(\go{}{d}{q})\).
\end{proposition}
\begin{proposition}
	Take \(q\) odd, \(s\) such that \(2^{s+1}\parallel q^2-1\), and \(r_t\) such that \(2^{r_t}\le n<2^{r_t+1}\), where \(n=2m\) or \(n=2m+1\) is the degree of the group. Then:
	\begin{align}
	\label{recappsp2}	\exp_2(\psp{2m}{q})&=
	\begin{dcases*}
	2^{s+r_t-1} & if \(m\ne2^k\)	\\
	2^{s+r_t-2} & if \(m=2^k\)
	\end{dcases*}
	\\
	\label{recapom1}	\exp_2(\om{}{2m+1}{q})=\exp_2(\pom{}{2m+1}{q})&=
	\begin{dcases*}
	2^{s+r_t-1} & if \(m\ne2^k\)	\\
	2^{s+r_t-2} & if \(m=2^k\)						
	\end{dcases*}
	\\
	\label{recapom2}	\exp_2(\om{\eta}{2m}{q})=\exp_2(\pom{\eta}{2m}{q})&=
	\begin{dcases*}
	2^{s+r_t-1} & if \(m\ne2^k\)	\\
	2^{s+r_t-2} & if \(m=2^k\)
	\end{dcases*}(m>2)
	\end{align}
\end{proposition}
\begin{proposition}
	Take \(p\mid q\), and \(G\) a symplectic or orthogonal group over the field \(\f{q}\). Then
	\begin{equation}\label{pq}
	\exp_p(G)=\min\set{p^a\mid p^a>c-1},
	\end{equation}
	where \(c\) is \(2m\) if \(G=\mathrm{(P)}\sp{2m}{q}\) or \(\mathrm{(P)}\om{}{2m+1}{q}\), and \(2m-2\) if \(G=\mathrm{(P)}\om{\pm}{2m}{q}\).
\end{proposition}
These calculations follow from a series of results describing the \psylows of the classical groups, proved in~\cite{testerman,weir,carterfong,wong}, which we recall here.

\subsection{\texorpdfstring{\psylows}{Sylow p-subgroups}in characteristic prime to \texorpdfstring{\(p\)}{p}}

The \psylows with \(q\) prime to \(p\) and \(p\ne2\), are described in~\cite{weir}. The construction holds for \(\sp{2m}{q}\) and \(\so{}{n}{q}\).

For the next results, take \(e=\ord{p}{q}\), and define \(r\) as \(q^e-1=p^rx\), where \((p,x)=1\).

\subsubsection{General Linear Group}\label{glps}
We need some information concerning the \psylows of the general linear group, since they are required in the description of the \psylows of the symplectic and orthogonal groups.

Consider \(\gl{n}{q}\). Suppose \(n=c+ea\) and \(a=a_0+a_1p+\ldots+a_vp^v\), where \(0\le c<e\) and \(0\le a_i<p\), and take \(G_0\) a \psylow of \(\gl{e}{q}\), which is a cyclic group of order \(p^r\).

For example, fix a base of \(\f{q^e}\) over \(\f{q}\), and identify \(\aut_{\f{q}}\f{q^e}\) with \(\gl{e}{q}\). The natural action of \(\f{q^e}^{\times}\) on \(\f{q^e}\) induces an embedding \(\f{q^e}^{\times}\hookrightarrow\gl{e}{q}\), then take a maximal \(p\text{-subgroup}\).

Define \(G_{i+1}=G_i\wr C_p\). Then a \psylow of \(\gl{n}{q}\) is isomorphic to \(\prod_0^vG_i^{a_i}\).

In particular, \(\exp_p(\gl{n}{q})=\exp(G_v)=p^{r+v}\), where \(ep^v\le n<ep^{v+1}\).

\subsubsection{Symplectic Group}\label{spps}
Consider \(\sp{2n}{q}\). If \(e\) is even, a \psylow of \(\sp{2n}{q}\) is already a \psylow of \(\gl{2n}{q}\). If \(e\) is odd, a similar construction gives that a \psylow is isomorphic to \(\prod_0^vG_i^{b_i}\), where \(2n=d+2be\) (\(0\le d<2e\)),  \(b=b_0+b_1p+\ldots+b_vp^v\) (\(0\le b_i<p\)), \(G_0\) is a \psylow of \(\sp{2e}{q}\), and \(G_{i+1}=G_i\wr C_p\).

\(G_0\) is again a cyclic group of order \(p^r\): consider the subgroup \(R\) of \(\sp{2e}{q}\) of all \(M=\begin{pmatrix}A & 0\\0 & B\end{pmatrix}\) where \(A\) and \(B\) belong to a \psylow of \(\gl{e}{q}\). Being symplectic implies that \(A^{\top}B=1\), so that \(R\) is isomorphic to a \psylow of \(\gl{e}{q}\).

In particular, \(\exp_p(\sp{2n}{q})=\exp(G_v)=p^{r+v}\), where \(ep^v\le2n<ep^{v+1}\) if \(e\) is even and \(2ep^v\le2n<2ep^{v+1}\) if \(e\) is odd.

\subsubsection{Orthogonal Group}

In odd dimension, the construction is almost the same as for the symplectic group (notice that \({\sp{2m}{q}}\) and \({\go{}{2m+1}{q}}\) have the same order), giving the same result in term of exponent evaluation: \(\exp_p(\so{}{2m+1}{q})=p^{r+v}\), where \(ep^v\le2m+1<ep^{v+1}\) if \(e\) is even and \(2ep^v\le2m<2ep^{v+1}\) if \(e\) is odd. In even dimension, consider \(\so{\epsilon}{2m}{q}\): a \psylow is already a \psylow of \(\so{}{2m+1}{q}\) if \(p\mid q^m-\epsilon\), and a \psylow of \(\so{}{2m-1}{q}\) otherwise.

\subsection{\texorpdfstring{\dsylows}{Sylow 2-subgroups}in odd characteristic}

The construction of \dsylows is described in~\cite{carterfong} and~\cite{wong}. The idea is the same as before, with proper adjustments.

\subsubsection{Symplectic Group}

Consider \(\sp{2}{q}\) first, and denote with \(W\) a \dsyloww. Since \(\size{\sp{2}{q}}=q\bigl(q^2-1\bigr)\), \(\size{W}=2^{s+1}\) where \(2^{s+1}\parallel q^2-1\).

If \(q\equiv1\bmod4\), let \(\epsilon\) be a primitive \(2^s\text{-th}\) root of unity in \(\f{q}\). Then
\[W\simeq\gen*{\begin{pmatrix}\epsilon & 0\\0 & \epsilon^{-1}\end{pmatrix},\begin{pmatrix}0 & 1\\-1 & 0\end{pmatrix}}.\]
If \(q\equiv3\bmod4\), let \(\epsilon\) be a primitive \(2^{s+1}\text{-th}\) root of unity in \(\f{q^2}\). Then
\[W\simeq\gen[\bigg]{{\begin{pmatrix}0 & 1\\1 & \epsilon+\epsilon^q\end{pmatrix}}^2,\begin{pmatrix}0 & 1\\-1 & 0\end{pmatrix}}.\]
By writing \(W\simeq\gen*{X,Y}\), it's clear that \(\exp(W)=\exp(X)=2^s\).

\begin{theorem}
	Let \(S\) be a \dsylow of \(\sp{2n}{q}\), and write \(2n=2^{r_1}+\ldots+2^{r_t}\), where \(r_1<\ldots<r_t\). Then \(S\simeq W_{r_1}\times\ldots\times W_{r_t}\), where \(W_r=\smash{W\wr\underbrace{C_2\wr\ldots\wr C_2}_{r-1}}\).
\end{theorem}
\smallskip
In particular, \(\exp_2(\sp{2n}{q})=\exp(W_{r_t})=2^{s+r_t-1}\).

A \dsylow \(S'\) of \(\psp{2n}{q}\) is obtained by quotienting a \dsylow \(S\) of \(\sp{2n}{q}\) over \(S\cap Z\) where \(Z=\gen*{-1_n}\) is the centre of the group.

If \(S\simeq W_{r_1}\times\ldots\times W_{r_t}\), then there exists an element of the form \((1_{2^{r_1}},\ldots,1_{2^{r_{t-1}}},g)\) of maximum order, whose powers meet the centre only in \(1_n\).

If \(S\simeq W_r\), notice that an element \(x\in W_r=W_{r-1}\wr C_2\) with maximum order \(2^c\) is of the form \(\bigl((g,h),\sigma\bigr)\), where \(gh\in W_{r-1}\) has maximum order and \(C_2=\gen*{\sigma}\), so that \(x^{2^{c-1}}=\bigl((-1_{2^{r-1}},-1_{2^{r-1}}),1\bigr)=-1_{2^r}\).

Therefore,
\begin{equation}\label{pspexp}
\exp_2(\psp{2n}{q})=
\begin{dcases*}
2^{s+r_t-1} & if \(n\ne2^k\)	\\
2^{s+r_t-2} & if \(n=2^k\)
\end{dcases*}
\end{equation}
where \(2^{s+1}\parallel q^2-1\), and \(2^{r_t}\le 2n<2^{r_t+1}\).

\subsubsection{Orthogonal Group}\label{omrep}

Consider the groups \(\so{}{2n+1}{q}\) and \(\go{\pm}{2n}{q}\) first.

Since \(\size{\go{\epsilon}{2}{q}}=2(q-\epsilon)\), again \(\size{W}=2^{s+1}\) where \(2^{s+1}\parallel q^2-1\). \(W\) is dihedral so \(\exp(W)=2^s\), in particular we may take \(W=\gen*{u,w}\), where:
\begin{itemize}
	\item if \(q\equiv1\bmod4\), then the underlying form is of plus type and is represented by \(Q(x_1,x_2)=2x_1x_2\). Take \(\varepsilon\) a primitive \(2^s\text{-th}\) root of unity in \(\f{q}\) and 
	\begin{equation*}\label{omrep1}
	u=\begin{pmatrix}\epsilon & 0\\0 & \epsilon^{-1}\end{pmatrix}\qquad
	w=\begin{pmatrix}0 & 1\\1 & 0\end{pmatrix};
	\end{equation*}
	\item if \(q\equiv3\bmod4\), then the underlying form is of minus type and is represented by \(Q(x_1,x_2)=x_1^2+x_2^2\). Take \(a,b\in\f{q}\) such that \(a+b\sqrt{-1}\) is a primitive \(2^s\text{-th}\) root of unity in \(\f{q^2}\) and 
	\begin{equation*}\label{omrep2}
	u=\begin{pmatrix}a & b\\-b & a\end{pmatrix}\qquad
	w=\begin{pmatrix}-1 & 0\\0 & 1\end{pmatrix}.
	\end{equation*}
\end{itemize}
An order comparison shows that a \dsylow of \(\so{}{3}{q}\) may be obtained with the natural embedding \(W\mapsto\begin{pmatrix}\det W & 0\\0 & W\end{pmatrix}\).
In general, the following holds:
\begin{theorem}
	Let \(S\) be a \dsylow of \(\so{}{2n+1}{q}\) and let \(2n=2^{r_1}+\ldots+2^{r_t}\), with \(r_1<\ldots<r_t\). Then \(S\simeq W_{r_1}\times\ldots\times W_{r_t}\), where \(\smash{W_r=W\wr\underbrace{C_2\wr\ldots\wr C_2}_{r-1}}\) and \(W\) is a \dsylow of \(\go{\eta}{2}{q}\) with \(q\equiv\eta\bmod4\).
\end{theorem}
While in even dimension:
\begin{theorem}\label{thgo}
	Let \(S\) be a \dsylow of \(\go{\eta}{2n}{q}\).
	\begin{enumerate}
		\item If \(q^n\equiv\eta\bmod4\), then \(S\) is isomorphic to a \dsylow of \(\so{}{2n+1}{q}\).
		\item If \(q^n\equiv-\eta\bmod4\), then \(S\simeq C_2\times C_2\times S_0\), where \(S_0\) is a \dsylow of \(\so{}{2n-1}{q}\).
	\end{enumerate}
\end{theorem}
The exponent of a \dsylow is the same in \(\go{}{2n}{q}\) and \(\so{}{2n}{q}\). The argument is similar to the one used to deduce \eqref{pspexp}: consider the previous construction \(W=\gen*{u,w}\) and check the determinant. If \(n=2^k\),  a \dsylow of \(\so{}{2}{q}\) is obtained using \(W'=\gen*{u}\) instead, which has the order halved but the same exponent. If \(n\ne2^k\), consider a \dsylow \((C_2\times C_2)\times W_{r_1}\times\ldots\times W_{r_t}\) of \(\go{}{2n}{q}\), then an element of maximum order and determinant one can be obtained by taking an element of \(W_{r_t}\) with maximum order, and adjusting the determinant by taking suitable elements of the other groups.

The comparison of the exponents of \(\so{}{n}{q}\) and \(\om{}{n}{q}\) shares the same idea, i.e. if \(n\ne 2^k\) then the exponent doesn't change: \(S\simeq (C_2\times C_2)\times W_{r_1}\times\ldots\times\ldots W_{r_t}\) so take an element of \(W_{r_t}\) with maximum order, and adjust the spinor norm taking a proper element of the other groups.

It's not immediately obvious what happens when \(n=2^r\) though, i.e. if there exists an element of maximum order and spinor norm 1. The result can be deduced anyway using theorems 7, 8, 10 in~\cite{wong}.

\begin{theorem}\label{thomeganonsquare}
	If \(q^n\equiv-\eta\bmod4\), then a \dsylow of \(\om{\eta}{2n}{q}=\pom{\eta}{2n}{q}\) is isomorphic to a \dsylow of \(\go{\eta'}{2n-2}{q}\), where \(q^{n-1}\equiv\eta'\bmod4\).
\end{theorem}
%Denote with \(\k{}{n}{q}\) the kernel of the spinor norm in \(\go{}{n}{q}\), so that \(\size{\go{}{n}{q}\colon\k{}{n}{q}}=2\) and \(\k{}{n}{q}\cap\so{}{n}{q}=\om{}{n}{q}\).
\begin{theorem}\label{thomegadispari}
	A \dsylow of \(\om{}{2n+1}{q}=\pom{}{2n+1}{q}\) is isomorphic to a \dsylow of \(\k{\eta}{2n}{q}\), where \(q^n\equiv\eta\bmod4\).
\end{theorem}

So the remaining cases are \(\k{\eta}{2n}{q}\)  and \(\om{\eta}{2n}{q}\), where \(q^n\equiv\eta\bmod4\) and \(n=2^r\).

Consider \(\k{\eta}{2n}{q}\) first. Take \(n=1\), and remember the construction of a \dsylow of \(\go{}{2}{q}\) as \(W\simeq\gen*{u,w}\), with \(u^{2^s}=w^2=1\) and \(u^w=u^{-1}\) (\ref{omrep}).

The \dsylow of \(\k{}{2}{q}\) contained in \(W\simeq\gen*{u,w}\) is \(W'\simeq\gen*{v,w}\), where \(v=u^2\). It's still dihedral, but it has both the order and the exponent halved, being now \(2^s\) and \(2^{s-1}\) respectively. In particular:
\begin{equation}\label{kgen}
v^{2^{s-1}}=w^2=1,\qquad v^w=v^{-1}.
\end{equation}
Put \(e=uw\), so that \(e\in\go{}{2}{q}\backslash\k{}{2}{q}\), \(e^2=uu^w=1\) and \(W\simeq W'\gen*{e}\), where
\begin{equation}\label{krel}
v^e=v^{-1},\qquad w^e=vw.
\end{equation}
From now on use \(e\) to denote \(\begin{pmatrix}e & 0\\0 & 1_{2n-2}\end{pmatrix}\in\go{}{2n}{q}\backslash\k{}{2n}{q}\), which has order 2.
\begin{theorem}
	Take \(\k{\eta}{2n}{q}\), where \(n=2^r\) (\(r\ge0\)) and \(q^n\equiv\eta\bmod4\). Let \(2^{s+1}\parallel q^2-1\), \(E=\gen*{e}\simeq C_2\), \(W'\simeq\gen*{v,w}\) as in \eqref{kgen}, \(\rho:E\rightarrow\aut W'\) as in \eqref{krel}, and \(\,V\!=\gen*{a,b}\simeq C_2\times C_2\).
	
	Then a \dsylow of \(\k{\eta}{2n}{q}\) is isomorphic to the repeated twisted wreath product \(W'_r=W'\twr\underbrace{V\twr\ldots\twr V}_{r}\), where the action of \(\,V\!\) on \((W'_i)^2\) is given by \((x,y)^a=(x^e,y^e)\) and \((x,y)^b=(y,x)\) for any \(x,y\in W'_i\).
\end{theorem}
With this action, \(\exp(W'_r)=2\cdot\exp(W'_{r-1})\), so that \(\exp(W'_r)=2^r\exp(W')=2^{r+s-1}\). Indeed, take for example \(v_i\in W'_i\) of maximum order, then \(\bigl((v_i,1),b\bigr)\in W'_{i+1}\) has double its order.
\br
Consider now \(\om{\eta}{2n}{q}\) and \(\pom{\eta}{2n}{q}=\om{}{2n}{q}/Z\), where \(Z=\gen*{z}\), and \(z=-1_{2n}\) is the centre of \(\om{\eta}{2n}{q}\), and again \(q^n\equiv\eta\bmod4\).

For \(n=1\) a \dsylow of \(\om{\eta}{2}{q}\) is \(W''=\gen*{v}\), which has order and exponent \(2^{s-1}\). This time, \(W\) is not a split extension of \(T\), so the inductive construction has to start from \(n=2\).

Consider again \(e=\begin{pmatrix}uw & 0\\0 & 1_{2n-2}\end{pmatrix}\in\go{\eta}{2n}{q}\backslash\k{\eta}{2n}{q}\), and take \[f=\begin{pmatrix}1_2 & 0 & 0\\0 & w & 0\\0 & 0 & 1_{2n-4}\end{pmatrix}\in\k{\eta}{2n}{q}\backslash\om{\eta}{2n}{q}.\]
\(F=\gen*{e,f}\) is then a non-cyclic group of order 4 with trivial intersection with \(\om{}{2n}{q}\). For \(n=2\), one may take \(W''=\gen*{d,g,h,k}\), where
\begin{equation*}
d=\begin{pmatrix}u & 0\\0 & u^{-1}\end{pmatrix},\qquad
g=\begin{pmatrix}u & 0\\0 & u\end{pmatrix},\qquad
h=\begin{pmatrix}0 & 1_2\\1_2 & 0\end{pmatrix},\qquad
k=\begin{pmatrix}0 & w\\w & 0\end{pmatrix}.
\end{equation*}
In particular,
\begin{gather}\label{omgen}
d^{2^{s-1}}=g^{2^{s-1}}=z,\qquad z^2=h^2=k^2=1,\qquad d^h=d^{-1},\qquad g^k=g^{-1},\\\nonumber
[d,g]=[d,k]=[h,g]=[h,k]=1.
\end{gather}
which means that \(W''\) is a central product of two dihedral groups of order \(2^{s+1}\), and \(\gen*{z}\) is the centre of \(W''\). Then, the action of \(F\) on \(W''\) is given by
\begin{gather}\label{omrel}
d^e=g^{-1},\qquad g^e=d^{-1},\qquad h^e=gk,\qquad k^e=dh,\\\nonumber
d^f=g,\qquad g^f=d,\qquad h^f=k,\qquad k^f=h.
\end{gather}
\begin{theorem}
	Take \(\om{\eta}{2n}{q}\), where \(n=2^r\) (\(r\ge1\)) and \(q^n\equiv\eta\bmod4\). Let \(2^{s+1}\parallel q^2-1\), \(F=\gen*{e,f}\simeq C_2\times C_2\), \(W''\simeq\gen*{d,g,h,k}\) as in \eqref{omgen}, \(\rho:F\rightarrow\aut W''\) as in \eqref{omrel}, and \(V=\gen*{a,b,c}\simeq C_2\times C_2\times C_2\).
	
	Then a \dsylow of \(\om{\eta}{2n}{q}\) is isomorphic to the repeated twisted wreath product \(W''_r=W''\twr\underbrace{V\twr\ldots\twr V}_{r-1}\), where the action of \(\,V\!\) on \((W''_i)^2\) is given by \((x,y)^a=(x^e,y^e)\), \((x,y)^b=(x^f,y^f)\), and \((x,y)^c=(y,x)\) for any \(x,y\in W''_i\).% A \dsylow of \(\pom{\eta}{2n}{q}\) is isomorphic to \(W''_r/Z\), where \(Z=\gen*{-1_{2n}}\) is the centre of \(\om{\eta}{2n}{q}\).
\end{theorem}
Again \(\exp(W''_r)=2\cdot\exp(W''_{r-1})\), so the exponent is \(\exp(W''_r)=2^{r-1}\exp(W'')=2^{r+s-2}\).

Here, passing from \(\om{}{2n}{q}\) to \(\pom{}{2n}{q}\) requires to replace \(W''\) with \(W''/\gen*{z}\) in the construction. This lowers the exponent when \(r=1\), but it's irrelevant since \(\pom{+}{4}{q}\) is not simple and \(\pom{-}{4}{q}\simeq\psl{2}{q^2}\) is already studied elsewhere. For \(r>1\) the exponent doesn't change, take for example \(r=2\): \(d\) and \(g\) have order \(2^{s-1}\) since in the projective group \(z=1\).

Then \(x=\bigl((d,1),ac\bigr)\in (W''/\gen*{z})_2\) has order \(2^s\), indeed:
\begin{align*}
x	&=\bigl((d,1),ac\bigr),	\\
x^2	&=\bigl((d,1)(1,d^e),1\bigr)=\bigl((d,g^{-1}),1\bigr),	\\
x^3	&=\bigl((d,1)(g^{-e},d^e),ac\bigr)=\bigl((d^2,g^{-1}),ac\bigr),	\\
x^4 &=\bigl((d,1)(g^{-e},(d^e)^2),1\bigr)=\bigl((d^2,g^{-2}),1\bigr),	\\
x^{2^{s-1}}	&=\bigl((d^{2^{s-2}},g^{-2^{s-2}}),1\bigr),	\\
x^{2^s}	&=\bigl((d^{2^{s-1}},g^{-2^{s-1}}),1\bigr)=\bigl((1,1),1\bigr).
\end{align*}

\subsection{\texorpdfstring{\psylows}{Sylow p-subgroups}in characteristic \texorpdfstring{\(p\)}{p}}\label{sectionunipotent}

This case is completely solved by corollary 0.5 of~\cite{testerman}, which in our situation can be simplified in the following way:
\begin{theorem}
	Let \(G\) be a classical group defined over a field of characteristic \(p\). Then the exponent of a \psylow of \(G\) is \(\min\Set{p^a|p^a>c-1}\), where \(c\) is the Coxeter number of \(G\).
\end{theorem}
Note that the original proof of this theorem depends upon the classification of the conjugacy classes of unipotent elements, which requires \(p\) to be ``good'' prime. The result still holds in general, thanks to the extension of Bala-Carter Theorem due to Duckworth~\cite{duckworth}.
\section{The infinite families}\label{casidifficili}

This section covers the remaining four infinite families \((G,M)\) where \(M=N_G(H)\) and \((G,H)\) are as in the following table:
\begin{table}[h]
	\centering
	\begin{tabular}{ccc}
		\toprule
		\(G\) & \(H\) & Remarks \\
		\midrule
		\(\psp{2m}{q}\)		&	\(\om{-}{2m}{q}\)	&	\(m,q\) even	\\
		\(\pom{}{2m+1}{q}\)	&	\(\om{-}{2m}{q}\)	&	\(m\) even, \(q\) odd	\\
		\(\pom{+}{2m}{q}\)	&	\(\om{}{2m-1}{q}\)	&	\(m\) even	\\
		\(\psp{4}{q}\)		&	\(\psp{2}{q^2}\)		&	\\
		\bottomrule	
	\end{tabular}
	\caption{The cases (a)-(d) of \autoref{table}.}
	\label{lietype}
\end{table}

The idea is to compare the exponents of the \psylows of \(G\) and \(H\) using the results in the previous section, and check \(\exp(M)\) when \(\exp(G)\ne\exp(H)\). The \psylows of \(G\) and \(H\) obtained with the constructions presented in the previous section will be called \(S_G\) and \(S_H\) respectively if there could be any ambiguity. Recall also the notation: if \(p\notdivides q\) then \(e=\ord{p}{q}\) and \(q^e=1+p^rx\) with \((x,p)=1\).

\begin{proposition}
	\(\psp{4}{q}\) contains a subgroup with the same exponent if and only if the characteristic of the underlying field is different from 3. An example of such a subgroup is \(\psp{2}{q^2}\rtimes C_2\).
\end{proposition}
\begin{proof}
	Consider first \(p\) odd, \(p\notdivides q\). Since \(p\ne2\), the degree of the groups is too small to require a wreath product in the construction of a \psyloww: using the notation in \ref{spps}, this means that they are of the form \(G_0\) or \(G_0^2\), so \(\exp_p(\psp{4}{q})=\exp_p(\psp{2}{q^2})=p^r\).
	
	Take now \(p=2\) in odd characteristic. By \eqref{pspexp} we have that \(\exp_2(\psp{4}{q})=2^a\)  with \(2^a\parallel q^2-1\), and \(\exp_2(\psp{2}{q^2})=2^{b-1}\) with \(2^b\parallel q^4-1\). But \(q^4-1=(q^2-1)(q^2+1)\), so \(b=a+1\).
	
	Finally, if \(p\mid q\) then by \eqref{pq}
	\begin{equation*}
	\exp_p(\psp{4}{q})=
	\begin{dcases*}
	\exp_p(\psp{2}{q^2}) & if \(p\ne2,\,3\) \\
	p\cdot\exp_p(\psp{2}{q^2}) & if \(p=2,\,3\)
	\end{dcases*}
	\end{equation*}
	This means that the only problem is when the characteristic of the field is 2 or 3, so we consider the normalizer \(\psp{2}{q^2}\rtimes C_2\), which is a maximal subgroup of \(\psp{4}{q}\).
	
	If \(q=3^k\), clearly \(\exp(\psp{2}{q^2})=\exp(\psp{2}{q^2}\rtimes C_2)=\frac{1}{3}\exp(\psp{4}{q})\).
	
	If \(q=2^k\) we have \(\exp(\psp{2}{q^2}\rtimes C_2)=2\cdot\exp(\psp{2}{q^2})=\exp(\psp{4}{q})\). Indeed, consider the automorphism of \(\f{q^2}\) of order 2 given by \(\sigma:x\mapsto x^{2^k}\). The action of \(C_2\) onto \(\psp{2}{q^2}\) is the induced automorphism, which maps each entry of the matrix with its \(2^k\text{-th}\) power.
	
	Recall that the upper unitriangular matrices form a \dsylow and take \(\alpha\in\f{q^2}\), so that \(x=\biggl(\begin{pmatrix}1 & \alpha\\0 & 1\end{pmatrix},\sigma\biggr)\in\psp{2}{q^2}\rtimes C_2\) has order 4 if \(\alpha\ne1\).
\end{proof}

\begin{proposition}
	\(\psp{2m}{q}\) with \(m\ge4\) even and \(q\) even always contains a subgroup with the same exponent. An example of such a subgroup is \(\om{-}{2m}{q}\).
\end{proposition}
\begin{proof}
	Consider \(p=2\): by \eqref{pq}, \(\exp_2(\psp{2m}{q})=\min\set{2^a\mid 2^a>2m-1}\) and \(\exp_2(\om{-}{2m}{q})=\min\set{2^a\mid 2^a>2m-3}\), so they're different iff \(2m-3<2^a\le 2m-1\) for some \(a\). This would imply \(2^{a-1}=m-1\), but \(m\) is even and \(m\ge4\) so it can't occur.
	
	Take now \(p\) odd, then \(S_G\) is isomorphic to a \psylow of \(\sp{2m}{q}\) and \(S_H\) to a \psylow of \(\so{-}{2m}{q}\). There are essentially two different situations, depending on \(\bigl(p,q^m+1\bigr)\).
	
	If \(p\mid q^m+1\), a \psylow of \(\so{-}{2m}{q}\) is isomorphic to a \psylow of \(\so{}{2m+1}{q}\):
	\begin{itemize}
		\item if \(e\) is even, \(S_H\) and \(S_G\) are \psylows of \(\gl{2m+1}{q}\) and \(\gl{2m}{q}\), respectively, so they are isomorphic, otherwise we would have \(\exp_p(\gl{2m+1}{q})<\exp_p(\gl{2m}{q})\). Indeed, notice that \(\size{\gl{2m+1}{q}}=\size{\gl{2m}{q}}\cdot(q^{2m+1}-1)\), but \(e\) is even so \(p\notdivides q^{2m+1}-1\), and the order of their \psylows is the same;
		\item if \(e\) is odd, \(S_H\) is isomorphic to a \psylow of \(\sp{2m}{q}\).
	\end{itemize}
	If \(p\notdivides q^m+1\), a \psylow of \(\so{-}{2m}{q}\) is isomorphic to a \psylow of \(\so{}{2m-1}{q}\):
	\begin{itemize}
		\item if \(e\) is even, \(S_H\) and \(S_G\) are \psylows of \(\gl{2m-1}{q}\) and \(\gl{2m}{q}\), respectively. Since \(\size{\gl{2m}{q}}=\size{\gl{2m-1}{q}}\cdot(q^{2m}-1)\), they have the same \(p\text{-power}\) part in the exponent, unless \(p\mid q^m-1\) and \(2m=ep^i\) (see \eqref{recapgl}). This can't happen, since \(p\mid q^m-1\) implies \(e\mid m\), then \(2(m/e)=p^i\) but \(p\) is odd;
		\item if \(e\) is odd, \(S_H\) is a \psylow of \(\sp{2m-2}{q}\). Like in the previous case, the exponent may be different if \(2m=2ep^i\), but \(m\) is even while both \(e\) and \(p\) are odd.\qedhere
	\end{itemize}
\end{proof}

\begin{proposition}\label{difficile3}
	\(\pom{}{2m+1}{q}\) with \(m\ge4\) even and \(q\) odd always contains a subgroup with the same exponent, unless \(p^a=2m-1\) for some \(a\), where \(p\) is the characteristic of the underlying field. An example of such a subgroup is \(\om{-}{2m}{q}\).
\end{proposition}
\begin{proof}
	\(\exp_2(\pom{}{2m+1}{q})=\exp_2(\om{-}{2m}{q})\) follows immediately from \eqref{recapom1} and \eqref{recapom2}.
	
	Take now \(p\) odd, \(p\notdivides q\). Since \(p\) is odd, \(S_G\) is isomorphic to a \psylow of \(\so{}{2m+1}{q}\) and \(S_H\) to a \psylow of \(\so{-}{2m}{q}\). If \(p\mid q^m+1\), a \psylow of \(\so{-}{2m}{q}\) is isomorphic to a \psylow of \(\so{}{2m+1}{q}\). If \(p\notdivides q^m+1\), it's isomorphic to a \psylow of \(\so{}{2m-1}{q}\).
	
	If \(e\) is odd, \(S_H\) and \(S_G\) are \psylows of \(\sp{2m-2}{q}\) and \(\sp{2m}{q}\), respectively, so they may not have the same exponent if \(2m=2ep^i\), but \(m\) is even while \(e\) and \(p\) are odd.
	
	If \(e\) is even, \(S_H\) and \(S_G\) are \psylows of \(\gl{2m-1}{q}\) and \(\gl{2m+1}{q}\), respectively. They may not have the same exponent if \(2m+1=ep^i+1\) (see \eqref{recapgl}), but this can't happen:
	\begin{itemize}
		\item if \(p\mid q^m-1\), \(e\mid m\) and \(p\) is odd;
		\item if \(p\notdivides q^m-1\), call \(e=2f\) so that \(m=fp^i\) and \(q^{2f}\equiv1\bmod p\). Now \(q^{m}\equiv q^{fp^i}\equiv q^f\equiv\pm1\bmod p\), which means that either \(p\mid q^m-1\) or \(p\mid q^m+1\), contradiction.
	\end{itemize}
	Finally, consider \(p\mid q\). By \eqref{pq}, \(\exp_p(\pom{}{2m+1}{q})=\min\Set{p^a|p^a>2m-1}\) and \(\exp_p(\om{-}{2m}{q})=\min\Set{p^a|p^a>2m-3}\), so they're different iff \(2m-3<p^a\le2m-1\), i.e. \(p^a=2m-1\), for some \(a\). If that happens, \(\exp_p(\pom{}{2m+1}{q})=p\cdot\exp_p(\om{-}{2m}{q})\).
	
	\(N_{\pom{}{2m+1}{q}}(\om{-}{2m}{q})=\k{-}{2m}{q}\). Indeed, it's known that \(\go{-}{2m}{q}\times\go{}{1}{q}\simeq\go{-}{2m}{q}\times C_2\) is a maximal subgroup of \(\go{}{2m+1}{q}\). Taking the kernel of the determinant and of the spinor norm, we get that \(\k{-}{2m}{q}=\om{-}{2m}{q}.2\) is a maximal subgroup of \(\om{}{2m+1}{q}=\pom{}{2m+1}{q}\), and the only one containing \(\om{-}{2m}{q}\).
	%see the classification of the maximal subgroups at~\cite[p. 94]{wilson}
	
	Since the formula in \eqref{pq} depends only on the ``type'' of group, in this case orthogonal, then \(\exp_p(\k{-}{2m}{q})=\exp_p(\om{-}{2m}{q})\) for \(p\mid q\), therefore \(\pom{}{2m+1}{q}\) doesn't contain a subgroup with the same exponent if \(p^a=2m-1\).
\end{proof}

\begin{proposition}
	\(\pom{+}{2m}{q}\) with \(m\ge4\) even always contains a subgroup with the same exponent. An example of such a subgroup is \(\om{}{2m-1}{q}\).
\end{proposition}
\begin{proof}
	If \(q\) is odd, \(\exp_2(\pom{+}{2m}{q})=\exp_2(\om{}{2m-1}{q})\) follows immediately from \eqref{recapom1} and \eqref{recapom2}.
	
	If \(p\mid q\), \(\exp_p(\pom{+}{2m}{q})=\exp_p(\om{}{2m-1}{q})\) follows from \eqref{pq} since \(c=2m-2\) for both groups.
	
	Take now \(p\) odd, \(p\notdivides q\). Since \(p\) is odd, \(S_G\) is isomorphic to a \psylow of \(\so{+}{2m}{q}\) and \(S_H\) to a \psylow of \(\so{}{2m-1}{q}\). If \(p\notdivides q^m-1\), a \psylow of \(\so{+}{2m}{q}\) is isomorphic to a \psylow of \(\so{}{2m-1}{q}\). If \(p\mid q^m-1\), it's isomorphic to a \psylow of \(\so{}{2m+1}{q}\), and the situation is the same as in the proof of \ref{difficile3}.
\end{proof}
\section{Alternating Groups}\label{alternante}

As described in~\cite{doctabella}, if the alternating group \(A_n\) contains a subgroup \(M\) with the same exponent, then \(A_k\unlhd M\le\left(S_k\times S_{n-k}\right)\cap A_n\). From now on, suppose \(k\ge n-k\), i.e. \(k\ge n/2\). Our claim is
\begin{proposition}
	The alternating group \(A_n\) (\(n\ge5\)) doesn't contain a subgroup with the same exponent iff either \(n=10\), \(n=p^r\) with \(p\) odd prime, or \(n=p_f+1\) where \(p_f\) is a Fermat prime.
	%or \(n=2^r+2\) and \(n=p^s+1\) with \(p\) odd prime.
\end{proposition}
Consider \(p\le n\) odd, then \(\exp_p(A_n)=p^t\), where \(p^t\le n<p^{t+1}\): take for example a \(p^t\)-cycle.

If \(n=p^t\) then \(\exp_p(M)\le\exp_p(S_k\times S_{n-k})=\exp_p(S_k)<p^t\), since \(k<n\). Conversely, if \(n\ne p^t\) we can take \(k=n-1\) and get \(\exp_p(A_{n-1})=\exp_p(A_n)\).

Consider then \(p=2\). Since a \(2^t\text{-cycle}\) doesn't belong to \(A_n\), an element of maximal even order needs at least another disjoint transposition, which implies that \(\exp_2(A_n)=2^t\) where \(2^t+2\le n<2^{t+1}+2\). If \(n\ne 2^t+2\) then \(\exp_2(A_{n-1})=\exp_2(A_n)\). Conversely, if \(n=2^t+2\) then take \(k=n-2\) and get \(\exp_2(A_n)=\exp_2((S_{n-2}\times S_2)\cap A_n)\).

Therefore:
\begin{itemize}
	\item if \(n\ne p^r\) and \(n\ne2^r+2\), then \(\exp(A_n)=\exp(A_{n-1})\);
	\item if \(n=p^r\), then \(A_n\) can't have a subgroup with the same exponent;
	\item if \(n=2^r+2\) and \(n\ne p^s+1\) for any \(p\) odd, then \(\exp(A_n)=\exp((S_{n-2}\times S_2)\cap A_n)\);
	\item if \(n=2^r+2\) and \(n=p^s+1\) for some \(p\) odd, then \(A_n\) can't have a subgroup with the same exponent, since \(M\le(S_{n-2}\times S_2)\cap A_n\) and \(M\le A_{n-1}\) imply \(M\le A_{n-2}\), hence \(\exp_q(M)<\exp_q(A_n)\) for both \(q=2\) and \(q=p\), as seen above.
\end{itemize}
In particular, \(A_n\) doesn't contain a subgroup with the same exponent if and only if either \(n=p^r\) with \(p\) odd prime, or \(n=2^r+2=p^s+1\) with \(p\) odd prime.

This last condition is realized when \(2^r+1=p^s\). If \(s=1\), then \(2^r+1\) is a prime number iff it's a Fermat prime (\(r=1\) is irrelevant since \(n\ge5\)). If \(s>1\), the only possibility is \(r=p=3\) and \(s=2\), i.e. \(n=10\), thanks to Mih\u{a}ilescu's theorem~\cite{mihailescu}.
\section{Other groups}\label{casifacili}

This section covers all the remaining pairs \((G,M)\) deduced from~\cite{doctabella}. Since here the possible \(M\) are explicitly listed, they can be studied computationally to get the following result:
\begin{proposition}
	Consider the pairs \((G,M)\) in \autoref{table} not labelled \emph{(a)-(e)}. The only pairs which have \(\exp(G)=\exp(M)\) are the following:
	\begin{equation*}
	(\hs,\m{22}),\qquad(\m{12},\m{11}),\qquad(\m{24},\m{23}).
	\end{equation*}
\end{proposition}

\begin{proof}
	Most of the cases can be easily computed (we used GAP), the results are shown in \autoref{expcalc} below.
	\begin{table}[ht]
		\centering
		\begin{tabular}{lr@{\hskip 7ex}lr}
			\toprule
			\(G\) & \(\exp(G)\) & \(M\) & \(\exp(M)\) \\
			\midrule
			\(A_6\)			&	60		&	\(\l{2}{5}\)	&	30		\\
			\(\uni{3}{3}\)	&	168		&	\(\l{2}{7}\)	&	84		\\
			\(\uni{3}{5}\)	&	840		&	\(A_7\)			&	420		\\
			\(\uni{4}{2}\)	&	180		&	\(S_6\)			&	60		\\
			\(\uni{4}{3}\)	&	2520	&	\(\l{3}{4}\)	&	420		\\
			&			&	\(A_7\)			&	420		\\
			\(\uni{5}{2}\)	&	3960	&	\(\l{2}{11}\)	&	330		\\
			\(\uni{6}{2}\)	&	27720	&	\(\m{22}\)		&	9240	\\
			\(\psp{4}{7}\)	&	4200	&	\(A_7\)			&	420		\\
			\(\sp{6}{2}\)	&	2520	&	\(S_8\)			&	840		\\
			\(\pom{+}{8}{2}\) &	2520	&	\(A_9\)			&	1260	\\
			\(G_2(3)\)		&	6552	&	\(\l{2}{13}\)	&	546		\\
			\(\prescript{2}{}{F_4(2)'}\)	&	3120	&	\(\l{2}{25}\)	&	780	\\
			\(\m{11}\)		&	1320	&	\(\l{2}{11}\)	&	330		\\
			\(\m{12}\)		&	1320	&	\(\m{11}\)		&	1320	\\
			\(\m{24}\)		&	212520	&	\(\m{23}\)		&	212520	\\
			\(\hs\)			&	9240	&	\(\m{22}\)		&	9240	\\
			\(\mcl\)		&	27720	&	\(\m{22}\)		&	9240	\\
			\(\co{2}\)		&	1275120	&	\(\m{23}\)		&	212520	\\
			\(\co{3}\)		&	637560	&	\(\m{23}\)		&	212520	\\	
			\bottomrule
		\end{tabular}
		\caption{Exponent evaluation for the remaining cases of \autoref{table}.}
		\label{expcalc}
	\end{table}
	Therefore \(\hs\), \(\m{12}\), \(\m{24}\) have a subgroup with the same exponent, and the others don't, since the listed subgroups are maximal.
	
	Notice that some pairs \((G,M)\) aren't included because \(M\) missing from the ATLAS database (i.e. there isn't an explicit set of generators to make computations with):
	\begin{equation*}
	(\uni{4}{2},2^4\rtimes A_5),\qquad(\l{6}{2},2^5\rtimes\l{5}{2}),\qquad(\pom{+}{8}{2},2^6\rtimes A_8).
	\end{equation*}
	They can all be excluded using the same argument: call \(M=2^k\rtimes H\), then \(\exp(G)=3\cdot n\cdot\exp(H)\) where \(n=1,2\). Then it's clear that \(\exp(M)\) is also missing a factor 3, since \(M\) is obtained by taking the semidirect product of \(H\) with a \(2\text{-group}\).
\end{proof}

\begin{help}
	I would like to thank my supervisor, Professor Andrea Lucchini, who inspired this work and helped me finding some key references.
\end{help}

\printbibliography\label{lastpage}

\end{document}